\documentclass{llncs}
\usepackage[latin9]{inputenc}
\usepackage{amsmath,amssymb,amscd}
\usepackage{amssymb}
\newif\ifcsl
\csltrue

\newif\ifmargin
\marginfalse

\newif\ifmarginB
\marginBfalse

\newif\ifskip
\skipfalse

\ifcsl
\else
\numberwithin{equation}{section} 
\numberwithin{figure}{section} 
\theoremstyle{plain}
\fi 
\newenvironment{renumerate}{\begin{enumerate}}{\end{enumerate}}

\ifcsl
\newtheorem{thm}{Theorem}
  \newtheorem{lem}[thm]{Lemma}
  \newtheorem{prop}[thm]{Proposition}
  \newtheorem{cor}[thm]{Corollary}
  \newtheorem{ex}[thm]{Example}

\else
\usepackage{babel}
\fi 

\begin{document}
\global\long\def\mc#1{\mathcal{#1}}
\global\long\def\angl#1{\left\langle #1\right\rangle }

\newcommand{\N}{\mathbb{N}}
\newcommand{\Z}{\mathbb{Z}}
\newcommand{\FOL}{\textsc{FOL}}
\newcommand{\SOL}{\textsc{SOL}}
\newcommand{\MSOL}{\textsc{MSOL}}
\newcommand{\TVOIMSOL}{\textsc{TVMSOL}}
\newcommand{\TVOISOL}{\textsc{TVSOL}}
\newcommand{\TVOI}{\textsc{TV}}

\title{Definability of Combinatorial Functions and Their Linear Recurrence Relations\\ \ \\
	{\normalsize Extended Abstract}}
\author{T. Kotek
\thanks{
Partially supported by a grant of the Graduate School of the Technion--Israel Institute of Technology}
\and
        J.A. Makowsky
\thanks{Partially supported by a grant of the Fund for
Promotion of Research of the Technion--Israel Institute of
Technology and grant ISF 1392/07 of the Israel Science Foundation (2007-2010)}
}
\institute{Department of Computer Science\\
Technion--Israel Institute of Technology\\
32000 Haifa, Israel\\
e-mail: \{tkotek,janos\}@cs.technion.ac.il
}
\maketitle
\begin{abstract}
We consider functions of natural numbers 
which allow a combinatorial interpretation as density functions (speed) of classes of relational structures, such as Fibonacci 
numbers, Bell numbers, Catalan numbers and the like.
Many of these functions satisfy a linear recurrence relation over 
$\mathbb Z$ or ${\mathbb Z}_m$ and allow an interpretation as counting the 
number of relations satisfying a property expressible 
in Monadic Second Order Logic (MSOL).

C. Blatter and E. Specker (1981) showed that
if such a function $f$  counts the number of binary
relations satisfying a property expressible in MSOL
then $f$ satisfies for every $m \in \mathbb{N}$
a linear recurrence relation over $\mathbb{Z}_m$.

In this paper we give a complete characterization in terms of definability
in MSOL of the combinatorial functions which satisfy a linear
recurrence relation over $\mathbb{Z}$, and discuss various
extensions and limitations of the Specker-Blatter theorem.
\end{abstract}

 \section{Introduction}
\label{intro}
\label{se:intro}

\subsection{The Speed of a Class of Finite Relational Structures}
Let $\mathcal{P}$ be a graph property,
and $\mathcal{P}^n$ be the set of graphs with vertex set $[n]$.
We denote by $sp_{\mathcal{P}}(n)= |\mathcal{P}^n|$ the number of labeled graphs in $\mathcal{P}^n$.
The function
$sp_{\mathcal{P}}(n)$ is called the 
{\em speed of $\mathcal{P}$},
or in earlier literature the 
{\em density of $\mathcal{P}$}.
Instead of graph properties we also study classes of finite relational structures $\mathcal{K}$
with relations $R_i: i = 1, \ldots, s$ of arity $\rho_i$.
For the case of $s=1$ and $\rho_1=1$ such classes can be identified with
binary words over the positions $1, \ldots, n$.

The study of the function
$sp_{\mathcal{K}}(n)$ has a rich literature concentrating on two types of behaviours of
the sequence
$sp_{\mathcal{K}}(n)$: 
\begin{itemize}
\item
Recurrence relations
\item
Growth rate
\end{itemize}
Clearly, the existence of recurrence relations limits the growth rate.

\begin{renumerate}
\item
In formal language theory it was studied in  N. Chomsky and M.P. Schuetzenberger
\cite{ar:ChomskySchuetzenberger}
who proved that for $\mathcal{K} =L$, a regular language, the sequence
$sp_{L}(n)$  satisfies a linear recurrence 
relation over $\mathbb{Z}$. This implies that
the formal power series
$\sum_n sp_{L}(n) X^n$  is rational.
The paper \cite{ar:ChomskySchuetzenberger} initiated the field Formal Languages and Formal Power Series.

Furthermore, it is known that $L$ is regular iff $L$ is definable in
Monadic Second Order Logic $\MSOL$, \cite{ar:buechi60}.

\item
In C. Blatter and E. Specker \cite{pr:BlatterSpecker81} the case of $\mathcal{K}$ was studied, where
$\rho_i \leq 2$ for all $i \leq s$ and $\mathcal{K}$ definable in $\MSOL$.
They showed that in this case for every $m \in \mathbb{N}$, 
the sequence $sp_{\mathcal{K}}(n)$  
is ultimately periodic modulo $m$, or equivalently,
that 
the sequence $sp_{\mathcal{K}}(n)$  
satisfies a linear recurrence relation over $\mathbb{Z}_m$.

\item
In E.R. Scheinerman and J. Zito \cite{ar:SZ94} the function
$sp_{\mathcal{P}}(n)$ was studied for {\em hereditary} graph properties $\mathcal{P}$,
i.e., graph properties closed under induced subgraphs.
They were interested in the growth properties of
$sp_{\mathcal{P}}(n)$. The topic was further developed by J. Balogh, B. Bollobas and D. Weinreich
in a sequence of papers, \cite{ar:BT95,ar:BBW00,ar:BBW01}, which showed that only six classes of growth
of
$sp_{\mathcal{P}}(n)$ are possible, roughly speaking, constant, polynomial, or exponential growth,
or growth in one of three
factorial ranges.
They also obtained similar results for monotone graph properties, i.e., graph properties
closed under subgraphs, \cite{ar:BBW02}. Early precursors of the study of
$sp_{\mathcal{P}}(n)$ for monotone graph properties is \cite{ar:EFR86},
and for hereditary graph properties, \cite{ar:PS92}.

We note that hereditary (monotone) graph properties $\mathcal{P}$ are characterized by a countable set 
$IForb(\mathcal{P})$ 
($SForb(\mathcal{P})$)
of forbidden induced subgraphs (subgraphs).
In the case that 
$IForb(\mathcal{P})$  is finite, $\mathcal{P}$ is definable in First Order Logic, $\FOL$,
and in the case that
$IForb(\mathcal{P})$
is  $\MSOL$-definable, also $\mathcal{P}$ is
$\MSOL$-definable.
The same holds also for monotone properties.

\item
The classification of the growth rate of
$sp_{\mathcal{P}}(n)$ was extended to minor-closed classes in
\cite{arXiv:BNW07}. 
We note that minor-closed classes $\mathcal{P}$ are always $\MSOL$-definable. This is due to
the Robertson-Seymour Theorem, which states that they are characterized by a finite set
$MForb(\mathcal{P})$
of forbidden minors.
\end{renumerate}

One common theme of all the above cited papers is the connection between
the definability properties of $\mathcal{K}$ and the arithmetic properties of
the sequence $sp_{\mathcal{K}}(n)$.  In this paper we concentrate on the relationship between definability of a class $\mathcal{K}$ of relational structures and the various linear recurrence relations $sp_{\mathcal{K}}(n)$ can satisfy. 

\subsection{Combinatorial Functions and Specker Functions}
We would like to say that 
a function $f: \N \rightarrow \N$ is a combinatorial function if
it has a combinatorial interpretation. 
One way of making this more precise is the following. We say that $\mc{K}$ is definable in  $\mathcal{L}$ if there is
a $\mathcal{L}$-sentence $\phi$ such that for every $\bar{R}$-structure $\mathfrak{A}$, 
$\mathfrak{A} \in \mc{K} $ iff $\mathfrak{A} \models \phi$.
Then a function $f: \N \rightarrow \N$ is a combinatorial function if
$f(n) = sp_{\mathcal{K}}(n)$ for some class of finite structures
$\mathcal{K}$ definable in a suitable logical formalism $\mathcal{L}$.
Here $\mathcal{L}$ could be $\FOL$, $\MSOL$ or any interesting fragment of Second Order Logic, $\SOL$. We assume the reader is familiar with these logics, cf. \cite{bk:EFT94}.

\begin{definition}[Specker\footnote{
E. Specker studied such functions in the late 1970ties in his
lectures on topology at ETH-Zurich.
} function]

A function $f: \N \rightarrow \N$ is called a {\em $\mathcal{L}^k$-Specker function}
if there is a finite set of relation symbols $\bar{R}$ of arity at most $k$ and a class
of $\bar{R}$-structures $\mc{K}$ definable in $\mathcal{L}$ such that $f(n) = sp_{\mc{K}}(n)$.
\end{definition}

A typical non-trivial example is given by A. Cayley's Theorem from 1889,
which says that $T(n)=n^{n-2}$ can be interpreted as
the number of labeled trees on $n$ vertices.
Another example are the Bell numbers $B_n$ which count the number of equivalence
relations on $n$ elements. 

In this paper we study under what conditions the Specker function given by
the sequence
$sp_{\mathcal{K}}(n)$ satisfies a linear recurrence relation.

\begin{ex}
\label{ex:1}\ 
\begin{renumerate}

\item
The number of binary relations on $[n]$ is 
$2^{n^2}$,
and the number of linear orders on $[n]$ is $n!$. Both are $\FOL^2$-Specker functions.
$n!$ satisfies the linear recurrence relation $n!= n \cdot (n-1)!$. We note the coefficient in the recurrence relation is not constant.

\item The Stirling numbers of the first kind denoted 
$\left[\substack{n\\ k} \right]$ 
are defined as the number of ways to arrange $n$ objects into $k$
cycles. It is well known that for $n >0$ we have
$\left[\substack{n\\ 1} \right] = (n-1)!$. 
Specker functions are functions in one variable.
For fixed $k$, 
$\left[\substack{n\\ k} \right]$ is a $\FOL^2$-Specker function.
Using our main results,
we shall discuss Stirling numbers in more detail in Section \ref{se:ex}, Proposition \ref{prop:Stirling2}
and Corollary \ref{cor:Stirling2}.
\item
\label{grow}
For the functions
$2^{n^2}, n^{n-2}$ and $n!$ no linear recurrence relation with constant coefficients exists, 
because functions defined by
linear recurrence relations with constant coefficients 
grow not faster than $2^{O(n)}$.
However, for every $m \in \N$ we have that
$2^{n^2}$ satisfies a linear recurrence relation over $\Z_m$, where the
coefficients depend on $m$.
\item The Catalan numbers $C_{n}$ count the number of valid arrangements
of $n$ pairs of parentheses. $C_{n}$ is even iff $n$ is not of
the form $n=2^{k}-1$ for some $k\in\mathbb{N}$ (\cite{bk:Knuth06}).
Therefore, the sequence $C_{n}$ cannot be ultimately periodic modulo $2$. We discuss the Catalan numbers in Section \ref{se:ex}.
\end{renumerate}
\end{ex}

For $R$ unary we can interpret $\left\langle [n],R\right\rangle$ 
as a binary word where position $i$ is occupied by letter $1$ if $i\in R$ 
and by letter $0$ otherwise. Similarly, For $\bar{R}=(R_1,\ldots,R_s)$ which 
consists of unary relations only we can interpret\\ $\left\langle[n],R_1,\ldots,R_s\right\rangle$ 
as a word over an alphabet of size $2^s$. 
With this way of viewing languages we have the celebrated theorem of
R. B\"uchi (and later  but independently of C. Elgot and B. Trakhtenbrot), cf. \cite{bk:Libkin04,bk:EF95} states:
 \begin{thm} 
 \label{th:buchi}
 Let $\mc{K}$ be a language. Then $\mathcal{K}$ is regular iff $\mc{K}'$ is definable in $\MSOL$ given the natural order $<_{nat}$ on $[n]$. 
\end{thm}

From Theorem \ref{th:buchi} and \cite{ar:ChomskySchuetzenberger} we get immediately:
\begin{prop}
\label{prop:CS}
If $f(n)=sp_{\mathcal{K}}(n)$ is definable in $\MSOL^1$, $\MSOL$ with unary relation symbols only, given the natural order $<_{nat}$ on $[n]$, then it satisfies 
a linear recurrence relation over $\Z$  
$$
sp_{\mc{K}}(n) = \sum_{j=1}^{d} a_j \cdot sp_{\mc{K}}(n-j)
$$
with constant coefficients, 
\end{prop}
We say a function $f:\N\to\N$ is ultimately periodic over
$\mathcal{R}=\Z$ or over $\mathcal{R}=\Z_m$ 
if there exist $i,n_0\in \N$ such that for every $n\geq n_0$, $f(n+i)=f(n)$ over $\mathcal{R}$. 
It is well-known that $f$ is ultimately periodic over $\Z_m$ iff 
it satisfies a linear recurrence relation with constant coefficients over $\Z_m$. 
We note that if $f$ satisfies a linear recurrence over $\Z$ then
it also satisfies a linear recurrence over $\Z_m$ for every $m$. 
C. Blatter and E. Specker proved  the following
remarkable but little known theorem in 
\cite{pr:BlatterSpecker81},\cite{pr:BlatterSpecker84},\cite{ar:Specker88}. 

\begin{thm}[Specker-Blatter Theorem] 
\label{th:BS}
If $f(n)=sp_{\mathcal{K}}(n)$ is definable in $\MSOL^2$, $\MSOL$ with
unary and binary relation symbols only, then
for every $m\in\mathbb{N},$ 
$f(n)$ satisfies a linear recurrence relation with constant coefficients
\[sp_{\mc K}(n)\equiv\sum_{j=1}^{d_{m}}a_{j}^{(m)}sp_{\mc K}(n-j)\,\,(mod\, m)\]
and hence is  ultimately periodic
over $\Z_m$.
\end{thm}
In \cite{ar:FischerMakowsky03}
it was shown that in Proposition \ref{prop:CS} and in Theorem \ref{th:BS}
the logic $\MSOL$ can be augmented by modular counting quantifiers.

Furthermore, E. Fischer showed in \cite{ar:Fischer02}
\begin{thm}
For every prime $p \in \N$ 
there is an $\FOL^4$-definable function $sp_{\mc{K}_p}(n)$, where $\mc{K}_p$ consists of finite $(E,R)$-structures
with $E$ binary and $R$ quaternary, which is not ultimately periodic 
modulo $p$.
\end{thm}

The definability status of various combinatorial functions from the literature
will be discussed in Section \ref{se:ex}.

\subsection{Formal Power Series}
Our main result can be viewed as related to 
the theory of generating functions for formal languages, cf. \cite{bk:SS78,bk:BR88}
Let $A$ be a commutative semi-ring with unity and denote by 
$A\left\langle\left\langle x \right\rangle\right\rangle$
the semi-ring of formal power series $F$ in one variable over $A$
$$
F=\sum_{n=0}^\infty f(n) x^n 
$$ 
where $f$ is a function from $\N$ to $A$. 
A power series $F$ in on variable is an 
{\em $A$-rational series} if it is in
the closure of the polynomials over $A$ by 
the sum, product and star operations, 
where the star operation $F^*$ is defined as $F^*=\sum_{i=0}^\infty F^i$.  

We say a function $f:\N\to A$ is $A$-rational if 
$\{f(n)\}_{n=0}^\infty$ is the sequence of coefficients of an $A$-rational series $F$. 
We will be interested in the cases of $A=\N$ and $A=\Z$. 
It is trivial that every $\N$-rational function is a $\Z$-rational function.
It is well-known that $\Z$-rational functions are exactly those 
functions $f:\N\to \Z$ which satisfy linear recurrence relations over $\Z$. 
Furthermore, $\Z$-rational functions can also be characterized 
as those functions $f$ which are the coefficents of 
the power series of $P(x)/Q(X)$, where $P,Q\in \Z[x]$ are polynomials and $Q(0)=1$. 

We aim to study Specker functions, which are by definition functions over $\N$. 
Clearly, every $\N$-rational function is over $\N$, 
while the $\Z$-rational functions may take negative values. 
Those non-negative $\Z$-rational functions 
which are  $\N$-rational  were characterized by Soittola, cf. \cite{ar:Soittola76}.
However, there are non-negative $\Z$-rational series which are not $\N$-rational, cf. \cite{bk:Eilenberg74,ar:BDFR01}.

There are strong ties between regular languages and rational series. From Theorem \ref{th:buchi} and 
\cite[Thm~II.5.1]{bk:SS78}
it follows that:
\begin{prop}
 \label{th:regularRational}
 Let $\mc{K}$ be a language. 
 If $\mc{K}$ is definable in $\MSOL$ given the natural order $<_{nat}$ on $[n]$, 
 then $sp_{\mc{K}}$  is $\N$-rational.
\end{prop}

\subsection{Extending $\MSOL$ and Order Invariance}
In this paper we investigate the existence of linear and modular linear
recurrence relations of Specker functions for the case
where $\mc{K}$ is definable in logics $\mathcal{L}$ which are sublogics of $\SOL$
and extend $\MSOL$.  

$C_{a,b}\MSOL$ is 
the extension of $\MSOL$ with modular counting quantifiers
"the number of elements $x$ satisfying $\phi(x)$ equals $a$ modulo $b$". 
$C_{a,b}\MSOL$ is a fragment of $SOL$ since 
the modular counting quantifiers are definble in $\SOL$ 
using a linear order of the universe which is existentially quantifed. 
\begin{ex}
The Specker function which counts the number of Eulerian graphs with $n$ vertices
is not $\MSOL$-definable. It is definable in $C_{a,b}\MSOL$ and indeed $b=2$ suffices. 
\end{ex}

We now look at the case where $[n]$ is equipped with a linear order.

\begin{definition}[Order invariance]
\ 
\begin{renumerate}
\item
A class $\mathcal{D}$ of {\em ordered $\bar{R}$-structures} is a class of $\bar{R}\cup\{<_1\}$-structures, where for every $\mathfrak{A}\in\mathcal{D}$ the interpretation of the relation symbol $<_1$ is always a linear order of the universe of $\mathfrak{A}$. 
\item
An $\mathcal{L}$ formula $\phi(\bar{R},<_1)$ for ordered $\bar{R}$-structures
is {\em truth-value order invariant (t.v.o.i.)}
if for any two structures $\mathfrak{A}_i= \langle [n], <_i, \bar{R} \rangle$ ($i=1,2$)
we have that $\mathfrak{A}_1 \models \phi$ iff
$\mathfrak{A}_2 \models \phi$.  Note $\mathfrak{A}_1$ and $\mathfrak{A}_2$ differ only in the linear orders $<_1$ and $<_2$ of $[n]$.
We denote by $\TVOI\mathcal{L}$ the set of $\mathcal{L}$-formulas
for ordered $\bar{R}$-structures
which are t.v.o.i. We consider $\TVOI\mc{L}$ formulas as formulas for $\bar{R}$-structures.
\item
For a class of ordered structures $\mc{D}$, let $osp_{\mc{D}}(n,<_1)=$
$$
| \{ (R_1, \ldots, R_s) \subseteq 
[n]^{\rho(1)} \times \ldots \times [n]^{\rho(s)}  : 
\langle [n], <_1, R_1, \ldots , R_s \rangle \in \mc{D} \}|
$$
A function $f:\N\to\N$ is called an  {\em $\mathcal{L}^k$-ordered Specker function} if there is a class of ordered $\bar{R}$-structures $\mc{D}$ of arity at most $k$ definable in $\mathcal{L}$ such that $f(n)=osp_{\mc{D}}$. 
\item
A function $f: \N \rightarrow \N$ is called a {\em counting order invariant (c.o.i.) $\mathcal{L}^k$-Specker function}
if there is a finite set of relation symbols $\bar{R}$ of arity at most $k$ and a class
of ordered $\bar{R}$-structures $\mc{D}$ definable in $\mathcal{L}$ such that 
for all linear orders $<_1$ and $<_2$ of $[n]$ we have
$f(n) = osp_{\mc{D}}(n, <_1) = osp_{\mc{D}}(n, <_2)$.
\end{renumerate}
\end{definition}

\begin{ex}
\label{ex:order}
\ 
\begin{renumerate}
\item
Every formula $\phi(\bar{R},<_1) \in \TVOISOL^k$ is equivalent to the formula 
$\psi(\bar{R}) = \exists <_1 \phi(\bar{R}, <_1) \land \phi_{linOrd}(<_1)  \in \SOL^k$, where $\phi_{linOrd}(<_1)$ says $<_1$ is a linear ordering of the universe.  
\item
Every $\TVOIMSOL^k$-Specker function is also a counting order invariant \\$\MSOL^k$-Specker function. 
\item
We shall see in Section \ref{se:modlinrec} that there are counting order invariant
$\MSOL^2$-definable Specker functions which are not $\TVOIMSOL^2$-definable.
\end{renumerate}
\end{ex}

The following proposition is folklore:
\begin{prop}
Every formula in $C_{a,b}\MSOL^k$ is equivalent to a formula in $\TVOIMSOL^k$.  
\end{prop}
\begin{proof}
We give a sketch of the proof for the $C_{a,b} \MSOL$ formula $\phi_{even}=C_{0,2}(x=x)$, which says the size of 
the universe is even. The general proof is similar. 
$\phi_{even}$ can be written as 
$
\phi(\bar{R},<_1)=\exists U \phi_{min}(U)\land
\forall x \forall y (\phi_{succ}(x,y) \to (x\in U \leftrightarrow y\notin U) )
\land \forall x (x\in U \to \exists y \ x<_1 y)
$
where 
$\phi_{min}(U)=\forall x 
\left((\neg\exists y \ y <_1 x)\to  x\in U\right)$ 
says the minimal element $x$ in the order $<_1$ belongs to $U$, and $\phi_{succ}(x,y)=(x<_1y)\land \neg \exists z (x<_1z \land z<_1 y)$ says $y$ is the successor of $x$ in $<_1$. 
\end{proof}

\subsection{Main Results}
Our first result is a characterization of functions over the natural numbers 
which satisfy a linear recurrence relation over $\Z$.

\begin{thm}
\label{th:main1}
Let $f$ be a function over $\mathbb{N}$. 
Then $f$ satisfies a linear recurrence relation over $\Z$ iff
$f=f_1 - f_2$ is the difference of two counting order invariant $\MSOL^1$-Specker
functions.
\end{thm}

In the terminology of rational functions we get the following corollary:
\begin{cor}
 \label{cor:main1}
 Let $f$ be a function $\mathbb{N}\to\N$. Then $f$ is $\Z$-rational iff $f$ is the difference of two $\N$-rational functions. 
\end{cor}

In the proof of Theorem \ref{th:main1}
we introduce the notion of Specker polynomials, which can be thought of as a special
case of graph polynomials where graphs are replaced by linear orders.

Next we show that the Specker-Blatter Theorem
cannot be extended to counting order invariant Specker functions which
are definable in $\MSOL^2$. More precisely:

\begin{prop}
\label{pr:E2eq}
Let $E_{2,=}(n)$ be the number of equivalence relations with two equal-sized equivalence classes. Then
$E_{2,=}(2n)= {\frac{1}{2}} {{2n}\choose{n}}$, and $E_{2,=}(2n+1)=0$.
$E_{2,=}$ is a counting order invariant $\MSOL^2$-definable. However, it does not satisfy a linear recurrence
relation over $\Z_2$, since it is not ultimately periodic modulo $2$.
To see this note that $E_{2,=}(2n)=0 \pmod{2}$ iff $n$ is an even power of $2$.
\end{prop}
In Section \ref{se:ex} we shall show in Corollary \ref{cor:Catalan} the same
also for the Catalan number.

However, if we require that the defining formula $\phi$ of a Specker function
is itself order invariant, i.e. $\phi \in \TVOIMSOL^2$, then the Specker-Blatter Theorem
still holds.

\begin{thm}
\label{th:mainSP}
Let $f$ be a $\TVOIMSOL^2$-Specker function.
Then, for all $m \in \N$ the function $f$ satisfies a modular linear recurrence
relation modulo $m$.
\end{thm}

Table \ref{table1} summarizes the relationship between definablity of a 
$\mc{L}^k$-Specker function $f(n)$ and existance of linear recurrence. 
We denote by $MLR$ that $f(n)$ has a modular linear recurrence 
(for every $m\in\mathbb{N})$ and by $LR$ that $f(n)$ satisfies a linear recurrence over $\Z$. 
We write $NO\,\,LR$ (respectively $NO\,\,MLR$) 
to indicate that there is some $\mc{L}^k$-Specker function without 
a linear recurrence over $\Z$ (respectively $\Z_m$, for some $m\in\N$). 
The entries in bold face are new.

\begin{table}[tpbh!]
{
\large
\begin{center}
\begin{tabular}{|c||c|c|c|c|}
\hline 
$k$ & $\MSOL^{k}$ & $C_{a,b}\MSOL^{k}$ & $\TVOIMSOL^{k}$ & c.o.i.$\MSOL^{k}$\tabularnewline
\hline
\hline 
4 &  $No\,\, MLR$ & $No\,\, MLR$ & $No\,\, MLR$ & $No\,\, MLR$\tabularnewline
\hline 
3 & ? & ? & ? & ?\tabularnewline
\hline 
2 & $\substack{MLR\\
\mathbf{(No\,\, LR)}}
$ & $\substack{MLR\\
\mathbf{(No\,\, LR)}}
$ & $\substack{MLR\\
\mathbf{(No\,\, LR)}}
$ & $\mathbf{No\,\, MLR}$ \tabularnewline
\hline 
1 & \multicolumn{4}{c|}{\textbf{All functions with $\mathbf{LR}$}}\tabularnewline
\hline
\end{tabular} \end{center} 
}
\caption{
Linear recurrences and definability of $\mc{L}^k$-Specker functions}
\label{table1}
\end{table}

\section{Linear Recurrence Relations for $\mc{L}^k$-Specker Functions}
\label{se:linrec}
To prove Theorem \ref{th:main1} we first introduce Specker polynomials and prove a generalized version of one direction of the theorem in subsection \ref{subse:Polynomials}. We finish this direction of the proof of Theorem \ref{th:main1} in subsection \ref{subse:Proof}. 
The other direction of Theorem \ref{th:main1} is easy and is also given in subsection \ref{subse:Proof}.  

\subsection{$\mc{L}^k$-Specker polynomials}
\label{subse:Polynomials}

\begin{definition} \
\begin{renumerate}
 \item
A $\mathcal{L}^k$-Specker polynomial $A(n,\bar{x})$ in indeterminate set $\bar{x}$
has the form
\[
\sum_{R_{1}:\Phi_{1}(R_{1})}
\cdots\sum_{R_{t}:
\Phi_{t}(R_{1},\ldots,R_{t})}
\left(
\prod_{v_{1}, \ldots, v_{k}:
\Psi_{1}(\bar{R},\bar{v})}x_{m_1}\cdots
\prod_{v_{1}, \ldots, v_{k}:
\Psi_{l}(\bar{R},\bar{v})}x_{m_l}
\right)
\]
where $\bar{v}$ stands for $(v_1,\ldots,v_k)$, $\bar{R}$ stands for $\left(R_{1},\ldots,R_{t}\right)$ and the $R_i$'s are relation variables of arity $\rho_i$ at most $k$.  The $R_i$'s range over relations of arity $k$ over $[n]$ and the $v_i$ range over elements of $[n]$ satisfying the iteration formulas  $\Phi_{i},\Psi_i \in\mathcal{L}$.
\item
Simple ordered $\mathcal{L}^k$-Specker polynomials and order invariance
thereof are defined
analogously to Specker functions.
\end{renumerate}
\end{definition}

Every Specker function can be viewed as a Specker polynomial in zero indeterminates.
Conversely,
if we evaluate a Specker polynomial at $x=1$ we get a Specker function.

In this subsection we prove a stronger version of Theorem \ref{th:main1}.
\begin{lem}
\label{lem:EvalPol} Let $A(n,\bar{z})$ be a c.o.i. $\MSOL^{1}$-Specker polynomial 
with indeterminates $\bar{z}=(z_{1},\ldots,z_{s})$ and let
$h_{1}(\bar{w}),\ldots,h_{s}(\bar{w})\in\mathbb{Z}\left[\bar{w}\right]$. 
Let \\ $A\left(n,(h_1(\bar{w}),\ldots,h_s(\bar{w}))\right)$ denote the variable subtitution in
$A(n,\bar{z})$ where for $i\in\left[s\right]$,  $z_{i}$ is substituted to $h_{i}(\bar{w})$.
Then $A\left(n,\bar{h}\right)$ is an integer evaluation of a
c.o.i. $\MSOL^{1}$-Specker polynomial.\end{lem}
\begin{proof}
We look at $A(n,\bar{z})$ with $z_{1}$ substituted to the polynomial
$$
h_{1}(\bar{w})=\sum_{j=1}^{d}c_{j}w_{1}^{\alpha_{j1}}\cdots w_{t}^{\alpha_{jt}}
$$
where $d,\alpha_{11},\ldots,\alpha_{dt}\in\mathbb{N}$ and $c_{1},\ldots,c_{d}\in\mathbb{Z}$.
The c.o.i. $\MSOL^{1}-$Specker polynomial $A(n,\bar{z})$ is given by
\[
\sum_{R_{1}:\Phi_{1}(R_{1})}\cdots\sum_{R_{m}:\Phi_{m}(R_{1},\ldots,R_{m-1})}\left(\prod_{v_{1}:\Psi_{1}(\bar{R},v_{1})}z_{1}\cdots\prod_{v_{1}:\Psi_{s}(\bar{R},v_{1})}z_{s}\right)
\]
so substituting $z_1$ to  $h_{1}(\bar{w})$ we get $A(n,(h_1(w),z_2,\ldots,z_s))=$
\[
\sum_{R_{1}:\Phi_{1}(R_{1})}\cdots\sum_{R_{m}:\Phi_{m}(R_{1},\ldots,R_{m-1})}\left(\prod_{v_{1}:\Psi_{1}(\bar{R},v_{1})}h_{1}(\bar{w})\cdots\prod_{v_{1}:\Psi_{s}(\bar{R},v_{1})}z_{s}\right).
\]
We note that for every $\alpha(v)\in\MSOL$ we can define an $\MSOL$ formula with $d$ unary relation
variables $\phi_{Part(\alpha)}(U_{1}\ldots,U_{d})$ which holds iff $U_{1},\ldots,U_{d}$
are a partition of the set of elements of $[n]$ which satisfy $\alpha(v)$. Then \\
$A\left(n,\left(h_{1}(\bar{w}),z_{2},\ldots,z_{s}\right)\right)=$
\begin{eqnarray*}
\sum_{R_{1}:\Phi_{1}(R_{1})}\cdots\
\sum_{R_{m}:\Phi_{m}(R_{1},\ldots,R_{m-1})}
\sum_{U_{1},\ldots,U_{d}:
\phi_{Part(\Psi_1)}(\bar{U})}
\left(\prod_{v_{1}:\Psi_{2}(\bar{R},v_{1})}z_{2}\cdots\right.\\
\left.\prod_{v_{1}:\Psi_{s}(\bar{R},v_{1})}z_{s}
\prod_{v_{1}:v_{1}\in U_{1}}c_{1}w_{1}^{\alpha_{11}}\cdots w_{t}^{\alpha_{1t}}\cdots\prod_{v_{1}:v_{1}\in U_{d}}c_{d}w_{1}^{\alpha_{d1}}\cdots w_{t}^{\alpha_{dt}}\right)
\end{eqnarray*}
Next, we note for any formula $\theta,$ \[
\prod_{v_{1}:\theta}c_{j}w_{1}^{\alpha_{j1}}\cdots w_{t}^{\alpha_{jt}}=\prod_{v_{1}:\theta}c_{j}\stackrel{\alpha_{j1}\mbox{ times}}{\overbrace{\prod_{v_{1}:\theta}w_{1}\cdots\prod_{v_{1}:\theta}w_{1}}}\cdots\stackrel{\alpha_{jt}\mbox{ times}}{\overbrace{\prod_{v_{1}:\theta}w_{t}\cdots\prod_{v_{1}:\theta}w_{t}}}\]
We now replace all $c_{j}$ with new indeterminates $w_{j}'$ and thus obtain that \\
$A\left(n,\left(h_{1}(\bar{w}),z_{2},\ldots,z_{s}\right)\right)$
is an evaluation of an c.o.i. $\MSOL^{1}$-Specker polynomial.

Doing the same for the other $z_{i}$ we get that $A\left(n,\left(h_{1}(\bar{w}),\ldots,h_{s}(\bar{w})\right)\right)$
is an evaluations of an o.i. $\MSOL^{1}$-definable Specker polynomial,
as required. \end{proof}
\begin{thm}
\label{th:main2}
 Let $A_n(\bar{x})$ 
be a sequence of polynomials with a finite indeterminate set
$\bar{x}=\left(x_{1},\ldots,x_{s}\right)$
which satisfies a linear recurrences over
$\mathbb{Z}$. Then there exists a c.o.i $\MSOL^1$-Specker polynomial $A'(n,\bar{x},\bar{y)}$
such that $A_n(\bar{x})=A'(n,\bar{x},\bar{a})$
where $\bar{a}=(a_{1},\ldots,a_{l})$ and $a_{i}\in\mathbb{Z}$ for
$i=1,\ldots,l$. \end{thm}
\begin{proof}
Let $A_n(\bar{x})$ be given by a linear recurrence\[
A_n(\bar{x})=\sum_{i=1}^{r}f_{i}(\bar{x})\cdot A_{n-i}(\bar{x}),\]
 where $f_{i}(\bar{x})\in\mathbb{Z}$$\left[\bar{x}\right]$ and initial
conditions $A_1(\bar{x}),\ldots,A_r(\bar{x})\in\mathbb{Z}\left[\bar{x}\right]$.
To write $A_n(\bar{x})$ as a c.o.i. $\MSOL^{1}$-Specker
polynomials, we sum over the paths of the recurrence tree. A path
in the recurrence tree corresponds to the successive application of
the recurrence $A_n(\bar{x})\to A_{n-i_{1}}(\bar{x})\to A_{n-i_1-i_{2}}(\bar{x})\to\cdots\to A_{n-i_1-\ldots-i_{l}}(\bar{x})$
where $i_{1},\ldots,i_{l}\in[r]$ and $A_{n-i_1-\ldots-i_{l}}(\bar{x})$ is an
initial condition. 

In the following, the $U_{i}$ for $i\in[r]$ stand for the vertices
in the path, $I_{i}$ for $i\in[r]$ stand for initial conditions
$A_i(\bar{x})$, and $S$ stands for all those elements of $[n]$ skipped by
the recurrence. We may write $A_n(\bar{x})$ as \[
A_n(\bar{x})=\sum_{\bar{U},\bar{I},S:\phi_{rec}(\bar{U},\bar{I},S)}\prod_{v:v\in U_{1}}f_{1}(\bar{x})\cdots\prod_{v:v\in U_{r}}f_{r}(\bar{x})\prod_{v:v\in I_{1}}A(1,\bar{x})\cdots\prod_{v:v\in I_{r}}A(r,\bar{x})\]
 where $\phi_{rec}(\bar{U},\bar{I},S)$ says 
\begin{itemize}
\item $\phi_{Part}(\bar{U},\bar{I},S)$ holds, i.e. $\bar{U},\bar{I},S$
is a partition of $[n]$,
\item $n\in\bigcup U_{i}$, i.e. the path in the recurrence tree starts
from $n$, 
\item $\left|\bigcup_{i=1}^{r}I_{i}\right|=1$, i.e. the path reaches exactly
one initial condition
\item if $v\in[n]-[r]$, then $v\notin\bigcup_{i=1}^{r}I_{i}$, i.e. the
path may not reach an initial condition until $v\in[r]$, 
\item if $v\in[r]$, then $v\notin\bigcup_{i=1}^{r}U_{i}$, i.e. the path
ends when reaching the initial conditions, and
\item for every $v\in U_{i}$, $\left\{ v-1,\ldots,v-(i-1)\right\} \subseteq S$
and $v-i\in\bigcup_{i=1}^{r}\left(U_{i}\cup I_{i}\right)$, i.e. the
next element in the path is $v-i$. 
\end{itemize}
The formula $\phi_{rec}$ is $\MSOL$ definable using the given order. Let $B(n,\bar{x})$
be \[
B(n,\bar{z})=\sum_{\bar{U},\bar{I},S:\phi_{rec}(\bar{U},\bar{I},S)}\prod_{v:v\in U_{1}}z_{1}\cdots\prod_{v:v\in U_{r}}z_{r}\prod_{v:v\in I_{1}}z_{r+1}\cdots\prod_{v:v\in I_{r}}z_{2r}\]
 then $B(n,\bar{z})$ is a c.o.i. $\MSOL^{1}$-Specker polynomial.
By Lemma \ref{lem:EvalPol}, substituting $z_{i}$ to $f_{i}(\bar{x})$
for $i\in[r]$ and to $A_{i-r}(\bar{x})$ for $i\in[2r]\backslash[r]$,
we have that $B\left(n,\left(f_{1}(\bar{x}),\ldots,f_{r}(\bar{x}),A(1,\bar{x}),\ldots,A(r,\bar{x})\right)\right)=A_n(\bar{x})$
is an evaluation in $\mathbb{Z}$ of a c.o.i. $\MSOL^{1}$-Specker polynomial.\end{proof}

\subsection{Proof of Theorem \ref{th:main1}}
\label{subse:Proof}

Let $f=f_1 - f_2$ and $f_1$ and $f_2$ be c.o.i $\MSOL^1$-Specker functions. 
By Proposition \ref{prop:CS} together with
Theorem \ref{th:buchi} we have that $f_1$ and $f_2$ satisfy
linear recurrence relations over $\mathbb{Z}$.
It is well known that finite sums, differences and products
of functions satisfying a linear recurrence relation
again satisfy a linear recurrence relation, cf. 
\cite[Chapter 8]{bk:LidlNiederreiter} or \cite[Chapter 6]{bk:Sidi03}.
Thus, $f =f_1 - f_2$  satisfies a linear recurrence relation over $\Z$. 

Conversely, if $f$ satisfies a linear recurrence relation over $\Z$, 
then by Theorem \ref{th:main2}, $f$ is given by an evaluation 
$\bar{a}=(a_1,\ldots,a_l)$ where $a_i\in\Z$ for $i=1,\ldots,l$ of a
c.o.i. $\MSOL^1$ Specker polynomial $A(n, \bar{y})$ in variables $y_i$.
We have to show that $f$ is a difference of two c.o.i. $\MSOL^1$-Specker functions. 
For the sake of simplicity we will show this only for the case of a 
$\MSOL^1$-Specker polynomial in one indeterminate,
$$
A(n,y)=\sum_{R:\Phi(R)}\prod_{v_1:\Psi(R,v_1)} y
$$
The general case is similar. 
We may write $A(n,y)$ as
\[
A(n,y)=\sum_{R,Y:\Phi(R)\land \Psi'(R,Y)}\prod_{v:v\in Y} y
\]
where  $\Psi'(Y)=\forall v \left( v\in Y \leftrightarrow \Psi(R,v) \right)$. 
For  $a>0$  we can write $\prod_{v:v\in Y} a$
as 
\[
\prod_{v:v\in Y} a=a^{|Y|}=
\left|\left\{ Z_{1},\ldots Z_{a}\mid  Z_{1},\ldots,Z_{a} \mbox{ form a partition of } Y   \right\} \right|.
\]
So, 
\[
A(n,a)=
\sum_{R,Y,\bar{Z} : \beta_a(R,Y,\bar{Z})} 1=
\left|\left\{R,Y,\bar{Z} \mid \beta_a(R,Y,\bar{Z})\right\}\right|
\]
where $\beta_a(R,Y,\bar{Z})=\Phi(R)\land \Psi'(R,Y)\land \phi_{part}(Y,\bar{Z})$ and $\phi_{part}(Y,Z_1,\ldots,Z_a)$ says $Z_1,\ldots.Z_a$ form a partition of $Y$.  We note that $\phi_{part}$ is definable in $\MSOL$. For $a=0$, 
\[
A(n,a)=\sum_{R:\gamma(R)} 1=\left| \left\{ R \mid \gamma(R)\right\}  \right|.,
\]
where $\gamma (R)=\Phi(R)\land \forall v_1 \neg \Psi(R,v_1)$. 
Thus, since the constant function $0$ is definable in $\MSOL$, we get that if $a\geq 0$ then $A(n,a)$ is the difference of two c.o.i $\MSOL^1$-Specker functions. 

For $a<0$ we have
\[
A(n,a)=\sum_{R,Y:\Phi(R)\land \Psi'(R,Y)}\prod_{v:v\in Y} |a| \prod_{v: v\in Y} (-1). 
\]
As above, we may write $A(n,a)$ as
\[
A(n,a)=\sum_{R,Y,\bar{Z}:\beta_{|a|}(R,Y,\bar{Z})} \prod_{v:v\in Y} (-1) 
\]
and we have
\[
A(n,a)=\sum_{R,Y,\bar{Z}:      \alpha_{Even}(Y) \land \beta_{|a|}({R,Y,\bar{Z}})} 1 -
       \sum_{R,Y,\bar{Z}: \neg \alpha_{Even}(Y) \land \beta_{|a|}({R,Y,\bar{Z}})} 1, 
\]
where $\alpha_{Even}(Y)$ says $|Y|$ is even. Thus, $A(n,a)$ is given by $A(n,a)=$
$$
\left|\left\{R,Y,\bar{Z} \mid  \alpha_{Even}(Y) \land \beta_{|a|}({R,Y,\bar{Z}})\right\}\right| -
\left|\left\{R,Y,\bar{Z} \mid   \neg\alpha_{Even}(Y) \land \beta_{|a|}({R,Y,\bar{Z}})\right\}\right|
$$
Since $\alpha_{Even}$ is definable in $\MSOL$ given an order, as discussed in example \ref{ex:order}, we get that $A(n,a)$ is a  difference of two c.o.i $\MSOL^1$-Specker functions for $a<0$.

\section{Modular Linear Recurrence Relations}
\label{se:modlinrec}

In this section we prove Theorem \ref{th:mainSP},
 the extension of the Specker-Blatter Theorem to $\TVOIMSOL^2$-Specker functions. 
We also prove Proposition \ref{pr:E2eq}, which shows Theorem \ref{th:mainSP} cannot 
be extended to c.o.i. $\MSOL^2$-Specker functions. 

\subsection{Specker Index}
We say a structure $\mc A=\left\langle [n],\bar{R},a\right\rangle $
is a pointed $\bar{R}$-structure if is consists of a universe $[n]$,
relations $R_{1},\ldots,R_{k}$, and an element $a\in [n]$ of the universe.
We now define a binary operation on pointed structures. Given two
pointed structures $\mc A_{1}=\left\langle [n_{1}],\bar{R}^{1},a_{1}\right\rangle $
and $\mc A_{2}=\left\langle [n_{2}],\bar{R}_{}^{2},a_{2}\right\rangle $,
let $Subst\left(\mc A_{1},\mc A_{2}\right)$ be a new pointed structure
$Subst(\mc A_{1},\mc A_{2})=\mc B$ where $\mc B=\left\langle [n_{1}]\sqcup[n_{2}]-\left\{ a_{1}\right\} ,\bar{R},a_{2}\right\rangle $,
such that the relations in $\bar{R}$ are defined as follows. For every
$R_{i}\in\bar{R}$ of arity $r$, $R_{i}=\left(R_{i}^{1}\cap\left([n_{1}]-\left\{ a_{1}\right\} \right)^{r}\right)\cup R_{i}^{2}\cup I$,
where $I$ contains all possibilities of replacing occurrences of
$a_{1}$ in $R^1_{i}$ with members of $[n_{2}]$. 

Similarly, we define $Subst(\mc A_{1},\mc A_{2})$ for a pointed structure
$\mc A_{1}$ and a structure $\mc A_{2}=\left\langle [n],\bar{R}\right\rangle $
(which is not pointed). Let $\mc C$ be a class of possibly pointed
$\bar{R}-$structures. We define an equivalence relation between $\bar{R}-$structures:
\begin{itemize}
\item We say $\mc A_{1}$ and $\mc A_{2}$ are equivalent, denoted $\mc A_{1}\sim_{Su(\mc C)}\mc A_{2}$
if for every pointed structure $\mc D$ we have that $Subst(\mc D,\mc A_{1})\in\mc C$
if and only if $Subst\left(\mc D,\mc A_{2}\right)\in\mc C$. 
\item The \emph{Specker index} of $\mc C$ is the number of equivalence
classes of $\sim_{Su(\mc C)}$. 
\end{itemize}
We use in the next subsection the following lemmas by Specker \cite{ar:Specker88}:
\begin{lem}
\label{lem:FiniteIndex} Let $\mc C$ be a class of $\bar{R}-$structures
of finite Specker index with all relation symbols in $\bar{R}$
at most binary. Then $f_{\mc C}(n)$ satisfies modular linear recurrence
relations for every $m\in\mathbb{N}.$
\end{lem}

\begin{lem}
\label{lem:MSOLFiniteIndex}If $\mc C$ is a class of $\bar{R}$-structures
which is $\MSOL^2-$definable, then $\mc C$ has finite Specker index. 
\end{lem}

\subsection{Proof of Theorem \ref{th:mainSP}}

We prove the following lemma:

\begin{lem}
\label{lem:TVOIMSOLFiniteIndex}If $\mc C$ is a class of $\bar{R}$-structures
which is  $\TVOIMSOL^2$-definable, then $\mc C$ has finite Specker index. 
\end{lem}
\begin{proof}
Let $\mc C$ be a set of $\bar{R}$-structures defined by the 
$\TVOIMSOL(\bar{R})$ formula $\phi$. Let $\mc C'$ be the class
of all $\bar{R}\cup R_{<}$-structures $\left\langle \mc A,R_{<}\right\rangle $
such that $\mc A\in\mc C$ and $R_{<}$ is a linear ordering of
the universe of $\mc A$. Let $\phi'$ be the $\MSOL(\bar{R}\cup \{R_{<}\})$ formula obtained from 
$\phi$ by the following changes:
\begin{enumerate}
\item the order used in $\phi$, $a <_1 b$, is replaced with the new relation symbol $R_{<}$ 
\item it is required that $R_{<}$ is a linear ordering of $[n]$. 
\end{enumerate}
We note that $\phi'$ defines $\mc C'$, since $\phi$ is truth-value order invariant
and that $\phi'$ is an $\MSOL^2$-formula.

We will now prove that $\mc{C}$ has finite Specker index, by showing that if it does not, then $\mc{C'}$ also has infinite Specker index, in contradiction to Lemma \ref{lem:MSOLFiniteIndex}.
Assume $\mc{C}$ has infinite Specker index. Then there is an infinite set $W$
of $\bar{R}-$structures, such that for every distinct $\mc A_{1},\mc A_{2}\in W$,
$\mc A_{1}\not\sim_{Su(C)}\mc A_{2}$. So, for every $\mc A_{1},\mc A_{2}\in W$
there is $\left\langle \mc G,s\right\rangle $ such that 
\[
Subst(\left\langle \mc G,s\right\rangle ,\mc A_{1})\in\mc{C}  \mbox{ iff } Subst(\left\langle \mc G,s\right\rangle ,\mc A_{2})\notin\mc C.
\]
Now look  at 
$W'=\left\{ \left\langle \mc A,R_< \right\rangle \mid\mc A\in W, R_< \mbox{ linear order of } [n] \right\}$,
where $[n]$ is the universe of $\mc{A}$. We note 
$Subst(\left\langle \mc G,R_{<_{\mc G}},s\right\rangle ,\left\langle \mc A_{1},R_{<_{{\mc A}_1}}\right\rangle )=
\left\langle Subst(\mc G,\mc A_{1}),R_{<^{'}}\right\rangle $,
where $R_{<^{'}}$ a linear ordering of the universe of $Subst(\mc G,\mc A_{1})$ which extends $R_{\mc{A}_1}$ and $R_{\mc{G}}$, 
and similarly for $\mc A_{2}$. 
Therefore, 
\[
Subst(\left\langle \mc G,R_{<_{\mc G}},s\right\rangle ,\left\langle \mc A_{1},R_{<_{{\mc A}_1}}\right\rangle )\in\mc C' 
\mbox{ iff }
Subst(\left\langle \mc G,R_{<_{\mc G}},s\right\rangle ,\left\langle \mc A_{2},R_{<_{{\mc A}_2}}\right\rangle
)\notin\mc C'.
\] 
So the Specker index of $\mc{C}'$ is infinite, in contradiction. 
\hfill $\Box$
\end{proof}

Theorem \ref{th:mainSP} now follows from lemma \ref{lem:FiniteIndex}.  

\subsection{Counting Order Invariant $\MSOL^2$}

Here we show the Specker-Blatter Theorem does not hold 
for c.o.i. $\MSOL^2$-definable Specker functions. 
We have two such examples, the function
$E_{2,=}$, as defined in Proposition \ref{pr:E2eq}, 
and the Catalan numbers, which we discuss in Section \ref{se:ex}.

More precisely, here we show:
\begin{prop}
$E_{2,=}$, as defined in Proposition \ref{pr:E2eq} is a c.o.i. $\MSOL^2$-Specker function.
\end{prop}

\begin{proof}
Let $\mc{C}$ be defined as follows:
\[
 \mc{C}=\left\{ \left\langle U,R,F \right\rangle \mid 
\left\langle [n],<_1,U,R,F \right\rangle\models \Phi \right\},
\]
where $U$ and $R$ are unary and $F$ is binary, $<_1$ is a linear order of $[n]$, and $\Phi$ is says 
\begin{enumerate}
\item $F$ is a function,
\item $U$ is the domain of $F$,
\item $R$ is the range of $F$, 
\item $U$ and $R$ form a partition of $[n]$,
\item the first element of $[n]$, is in $U$, 
\item $F:U\to R$ is a bijection, and
\item $F$ is monotone with respect to $<_1$. 
\end{enumerate}
We note $\mc{C}$ is $\MSOL^2$ definable. 
We note also that $osp_{\mc{C}}(n,<_1)$ is counting order invariant. 
$osp_{\mc{C}}(n,<_1)$ counts the number of partitions of $[n]$ into two equal
parts, because there is exactly one monotone bijection between any
two subsets of $[n]$ of equal size. The condition that $1\in U$
assures that we do not count the same partition twice. 
So $osp_\mc{C}(n,<_1) = E_{2,=}(n)$. 
\hfill $\Box$
\end{proof}
We know that  $E_{2,=}$ is not ultimately periodic modulo $2$ and hence the Specker-Blatter theorem cannot be extended to c.o.i. $\MSOL^2$-Specker functions.

\section{Examples}
\label{se:ex}
\subsection{Examples of $\MSOL^k$-Specker functions}
\subsubsection{Fibonacci and Lucas Numbers}

The Fibonacci numbers $F_{n}$ satisfy the linear recurrence
$F_{n}=F_{n-1}+F_{n-2}$
for $n>1$, $F_{0}=0$ and $F_{1}=1$.
The Lucas numbers $L_{n}$, a variation of the Fibonacci numbers,
satisfy the same recurrence for $n>1$, $L_{n}=L_{n-1}+L_{n-2}$,
but have different initial conditions, $L_{1}=1$ and $L_{0}=2$.

It follows from the proof of Theorem \ref{th:main1} that a function which satisfies a linear recurrence relation over $\N$ is a c.o.i $\MSOL^1$-Specker function. Thus. 
The Fibonacci and Lucas numbers are natural examples of c.o.i-$\MSOL^1$-Specker functions.

\subsubsection{Stirling Numbers}
The Stirling numbers of the first kind, 
denoted $\left[\substack{n\\ r} \right]$ 
are defined as the number of ways to arrange $n$ objects into $r$
cycles. For fixed $r$, this is an $\MSOL^2$-Specker function, since for
$E\subseteq[n]^{2}$ and $U\subseteq E$, the property that $U$ is
a cycle in $E$ and the property that $E$ is a disjoint union of
cycles are both $\MSOL^2$-definable.
Using again the growth argument from Example \ref{ex:1}(\ref{grow}),
we can see that the Stirling numbers of the first kind do not satisfy
a linear recurrence relation, because
$\left[\substack{n\\ 1} \right]$ 
grows like the factorial $(n-1)!$.
However, from the Specker-Blatter Theorem it follows that 
they satisfy a modular linear
recurrence relation for every $m$.

The Stirling numbers of the second kind, 
denoted $\left\{\substack{n\\ r} \right\}$,
count the number of partitions of a set $[n]$ into $r$ many non-empty subsets.
For fixed $r$, this is $\MSOL^2$-definable: We count the number of equivalence
relations with $r$ non-empty equivalence classes.
From the Specker-Blatter Theorem  it follows that
they satisfy a modular linear
recurrence relation for every $m$.
We did not find in the literature a linear recurrence relation
for the Stirling numbers of the second kind which fits our context.
But we show below that such a recurrence relation exists.

\begin{prop}
\label{prop:Stirling2}
For fixed $r$, the Stirling numbers of the second kind
are c.o.i. $\MSOL^1$-Specker functions.
\end{prop}
\begin{proof}
We use $r$ unary relations $U_1, \ldots , U_r$ and say that they
partition the set $[n]$ into non-empty sets.
However, when we permute the indices of the $U_i$'s
we count two such partitions twice. To avoid this
we use a linear ordering on $[n]$ and require that, with respect to this ordering,
the minimal element in $U_i$ is smaller than all the minimal elements
in $U_j$ for $j > i$.
\hfill $\Box$
\end{proof}

\begin{cor}
\label{cor:Stirling2}
For every $r$ there exists a linear recurrence relation with constant coefficients
for the Stirling numbers of the second kind
$\left\{\substack{n\\ r} \right\}$.
Further more there are constants $c_r$ such that
$\left\{\substack{n\\ r} \right\} \leq 2^{c_r \cdot n}$.
\end{cor}
Our proof is not constructive, and we did not bother here to calculate the
explicit linear recurrence relations or the constants $c_r$ for each $r$.

\subsubsection{Catalan Numbers}

\label{subse:catalan}
Catalan numbers were defined in Section \ref{se:intro} Example \ref{ex:1}.
We already noted that they do not satisfy any modular linear recurrence relation.
However, like the example $E_{2,=}$, the functions $f_c(n)=C_n$ is a c.o.i. $\MSOL^2$-Specker function.
To see this we use the following interpretation of Catalan numbers given in
\cite{bk:GKP94}.

$C_n$ counts the number of tuples $\bar{a}=(a_0,\ldots,a_{2n-1})\in[n]^{2n}$ such that
\begin{renumerate}
 \item
 $a_0=1$
 \item
 $a_{i-1}-a_i\in \{1,-1\}$ for $i=1,\ldots,2n-2$
 \item
 $a_{2n-1}=0$
\end{renumerate}
We can express this in $\MSOL^2$ using a linear order and
two unary functions. 
The two functions 
$F_1$ and $F_2$ are used to describe 
$a_0,\ldots,a_{n-1}$ and $a_n,\ldots,a_{2n-1}$ respectively. 
Let $\Phi_{Catalan}$ be the formula that says:
\begin{renumerate}
 \item $F_1,F_2:[n]\to[n]$
 \item $F_i(x+1)=F_i(x)\pm 1$ for $i=1,2$ and there exists $y=x+1\in[n]$
 \item  $F_1(n-1)=F_2(0)\pm 1$.
 \item $F_1(0)=1$
 \item  $F_2(n-1)=0$
 \end{renumerate}
The resulting formula is not t.vo.i., but $C_n = sp_{\mc{C}}(n)$ where 
$$
\mc{C}=\left| \left\{  (F_1,F_2) \mid \left\langle[n],<_1,F_1,F_2 \right\rangle \models \Phi_{Catalan} \right\} \right|
$$
is a c.o.i $\MSOL^2$-Specker function.

\begin{cor}
\label{cor:Catalan}
The function $f(n)=C_n$ is a c.o.i $\MSOL^2$-Specker function
and does not satisfy a modular linear recurrence relation modulo $2$.
\end{cor}

\subsubsection{Bell Numbers}
The Bell numbers $B_n$ count the number of equivalence relations on $n$ elements. We note $f(n)=B_n$ is a $\MSOL^2$-Specker function. However, $B_n$ is not c.o.i $\MSOL^1$-definable due to a growth argument.

\subsection{Examples of $\MSOL^k$-Specker Polynomials}
Our main interest are $\mc{L}^k$-Specker functions, and the $\mc{L}^k$-Specker polynomials were introduced as an auxiliary tool. However, there are natural examples in the literature of $\mc{L}^k$-Specker polynomials. 
\subsubsection{Fibonacci, Lucas and Chebyshev Polynomials}
The recurrence $F_{n}(x)=x\cdot F_{n-1}(x)+F_{n-2}(x)$, $F_{1}(x)=1$
and $F_{2}(x)=x$ defines the Fibonacci polynomials. The Fibonacci
numbers $F_{n}$ can be obtained as an evaluation of the Fibonacci polynomial for $x=1$,
$F_{n}(1)=F_{n}$. The Lucas polynomials are defined analogously.

The Chebyshev polynomials of the first kind (see \cite{bk:MasonHandscomb05}) are
defined similarly by the recurrence relation $T_{n+1}(x)=2xT_{n}(x)-T_{n-1}(x)$,
$T_{0}(x)=1,$ and $T_{1}(x)=x$. 
The Fibonacci, Lucas and Chebyshev polynomials are natural examples of Specker polynomials.
As they are defined by linear recurrence relations, they are c.o.i $\MSOL^1$-definable.

\subsubsection{Touchard Polynomials}
The Touchard polynomials are defined 
$$
T_{n}(x)=\sum_{k=1}^{n}\left\{ \substack{n\\
k}
\right\} x^{k}
$$ 
where $\left\{ \substack{n\\ k} \right\} $
is the Stirling number of the second kind. 
$T_n(x)$ is c.o.i $\MSOL^2$-definable;
To see this we note that it is defined by
\[
T_n(x)=\sum_{E:\Phi_{cliques}(E)}\,\,\prod_{u:\Phi_{first-in-cc}(E,u)}x
\]
where $\Phi_{cliques}(E)$ says $E$ is a disjoint union of cliques
and where 
$$\Phi_{first-in-cc}(E,u)=\forall v\left(\left(v <_1 u\,
\land\, v\not=u\right)\to(v,u)\notin E\right),$$
i.e. it says $u$ is the first vertex in its connected component,
with respect to the order (less or equal) of $[n]$. 
Clearly, 
$\Phi_{cliques}(E)$  and
$\Phi_{first-in-cc}(E,u)$
are in $\MSOL^2$.
We note that
$\Phi_{first-in-cc}(E,u)$ is not
invariant under the order $<_1$. 
The Bell numbers $B_n$ are given as an evaluation of $T_n(x)$, $B_n=T_n(1)$, which implies 
$T_n(x)$ is not co.i $\MSOL^1$-definable due to a growth argument.

\subsubsection{Mittag-Leffler Polynomials}
The Mittag-Leffler polynomial (see \cite{ar:Bateman40}) is given
by 
$$
M_{n}(x)=\sum_{k=0}^{n}
\left(\substack{n\\ k} \right)
(n-1)^{\underline{n-k}}2^{k}x^{\underline{k}}
$$ 
It holds that
\[
M_{n}(x)=\sum_{U\subseteq[n]}\left(n-1\right)
\cdots k\cdot2^{k}\cdot x\cdots(x-(k-1))=\sum_{U,F,T,S:\Phi_{M}}1.
\]
where $\Phi_{M}(U,F,T,S)$ says $U\subseteq[n],$ $F$ is an injective
function from $[n]-\left\{ n\right\} $ to $[n]$, $T$ is an injective
function from $U$ to $[x]$, and $S$ is a function from $U$ to
$\left\{ 1,n\right\} $. 
So, every evaluation of 
$M_{n}(x)$ where $x=m$, $m \in \N$, is
a c.o.i $\MSOL^2$-Specker function.

Note that
\begin{equation}
\label{eq:ML}
M_{n+1}(x) = \frac{1}{2}x \left[ M_n(x+1) + 2 M_n(x) + M_n(x-1) \right]
\end{equation}
This looks almost like a linear recurrence relation combined with
an interpolation formula, and is not of the kind
we are discussing here.

\section{Conclusions and Open Problems}

\label{se:conclu}

We have introduced the notion of one variable $\mc{L}^k$-Specker functions $f:\N\to\N$ 
as the speed of a $\mc{L}^k$-definable class of relational structures $\mc{K}$, i.e. $f(n)=sp_{\mc{K}}(n)$. 
We have investigated for which fragments $\mc{L}$ of $\SOL$ the $\mc{L}^k$-Specker functions satisfy linear recurrence relations over $\Z$ or $\Z_m$. 

We have used order invariance, definability criteria and limitation on the vocabulary
to continue the line of study, initiated by C. Blatter and E. Specker
\cite{pr:BlatterSpecker81},\cite{pr:BlatterSpecker84},\cite{ar:Specker88},
what type of linear recurrence relations one can expect from Specker functions.
We have completely characterized (Theorem \ref{th:main1}) the combinatorial functions which satisfy
linear recurrence relations with constant coefficients,
and we have discussed (Table \ref{table1} in Section \ref{se:intro}) 
how far one can extend the Specker-Blatter Theorem
in terms of order invariance and $\MSOL$-definability.
As a consequence, we obtained (Corollary \ref{cor:main1}) 
a new characterization of the $\Z$-rational functions 
$f:\N\to\N$ as the difference of $\N$-rational functions. 

We have not studied many variables $\mc{L}^k$-Specker functions arising from many-sorted structures, 
although this is a natural generalizations: For a class of graphs $\mc{K}$, $sp_{\mc{K}}(n,m)$ couns the number of 
graphs with $n$ vertices and $m$ edges which are in $\mc{K}$. 
Even for functions in one variable the following remain open:

\begin{renumerate}
\item
Can one prove similar theorems for linear recurrence relations
where the coefficients depend on $n$?
\item
Can one characterize the $\mc{L}^k$-Specker functions which satisfy
modular recurrence relations with constant coefficients for each modulus $m$,
i.e., is there some kind of a converse to Theorem \ref{th:mainSP}?
\item
Does Theorem \ref{th:mainSP} hold for $\TVOIMSOL^3$?
\end{renumerate}

Finally, for many-sorted $\mc{L}^k$-Specker functions studying both growth rate
and recurrence relations seems a promising topic of further research. 


\end{document}